\documentclass[a4paper,12pt,reqno]{amsart}

 \usepackage[T2A]{fontenc}
 \usepackage[utf8]{inputenc}
 \usepackage[ukrainian,english]{babel}
 \usepackage{amssymb,upref}
\usepackage{amsthm}

\usepackage{graphicx}
\usepackage{mathrsfs}

\usepackage[text={130mm,200mm}]{geometry}

\theoremstyle{plain}
\newtheorem{theorem}{Theorem}
\newtheorem*{theorem*}{Theorem}
\newtheorem{corollary}{Corollary}
\newtheorem*{corollary*}{Corollary}
\newtheorem{lemma}{Lemma}
\newtheorem*{lemma*}{Lemma}

\newtheorem*{proposition*}{Proposition}

\newtheorem*{conjecture*}{Conjecture}
\theoremstyle{definition}
\newtheorem{definition}{Definition}
\newtheorem*{definition*}{Definition}
\theoremstyle{remark}

\newtheorem*{remark*}{Remark}

\begin{document}

\title[Linear fractals of the Besicovitch--Eggleston type]{Linear fractals of the Besicovitch--Eggleston type}

\author{M. V. Pratsiovytyi}
\address[M. V. Pratsiovytyi]{Institute of Mathematics of NAS of Ukraine,  Dragomanov Ukrainian State University, Kyiv, Ukraine\\
ORCID 0000-0001-6130-9413}
\email{prats4444@gmail.com}
\author{S. O. Klymchuk}
\address[S. O. Klymchuk]{Institute of Mathematics of NAS of Ukraine, Kyiv, Ukraine\\
ORCID 0009-0005-3979-4543}
\email{svetaklymchuk@imath.kiev.ua}



\thanks{Scientific Journal of Drahomanov National Pedagogical University. Series 1. Physical and Mathematical Sciences. -- Kyiv: Drahomanov National Pedagogical University, 2012, No. 13(2), pp. 80--92.}

\begin{abstract}
We study topological, metric and fractal properties of set of numbers $[0;1]$ with given asymptotic mean of digits in their ternary representation. We investigate connection of these numbers and numbers with a given frequency of digits.
\end{abstract}

\maketitle
\section{Introduction}
It is well known that for each real number $x \in [0,1]$ there exists a sequence $(\alpha_n)$ such that $\alpha_n \in \mathcal{A} = {0,1,\ldots,s-1}$ and
$$
 x=\displaystyle\frac{\alpha_1}{s}+
   \displaystyle\frac{\alpha_2}{s^2}+\cdots+
   \displaystyle\frac{\alpha_n}{s^n}+\cdots\equiv\Delta^s_{\alpha_1\alpha_2\ldots\alpha_n\ldots}.
$$
The latter notation is called the \emph{$s$--adic representation}. We denote by $\alpha_k=\alpha_k(x)$ the \emph{$k$--th $s$--adic digit of the number} $x$. In general we cannot define the $k$--th digit uniquely as a function of the number, since the following equality holds
$$
\Delta^s_{c_1\ldots c_{k-1}c_k(0)}=\Delta^s_{c_1\ldots c_{k-1}[c_k-1](s-1)},
$$
where $(i)$ denotes the period in the expansion of the number. These numbers are called \emph{$s$--adic rational} and each such number admits exactly two distinct $s$--adic representations. All other numbers admit a unique expansion, and we call them \emph{$s$--adic irrational}. We agree to use only the first $s$--adic representation, namely the one with period $(0)$. Under this agreement, the function $\alpha_n(x)$ is well defined on the interval $[0,1]$.

Let $N_i(x,k)$ denote the quantity of digits $i \in \mathcal{A} = {0,1,\ldots,s-1}$ in the $s$--adic representation $\Delta^s_{\alpha_1\alpha_2\ldots\alpha_k\ldots}$ of a number $x \in [0,1]$ up to and including the $k$--th position, that is
$$
 N_i(x,k)=\# \{j:\, \alpha_j (x)=i, \, j\leqslant k\}.
$$
The number $k^{-1}N_i(x,k)$ is called the relative frequency of the digit $i$ in the $s$--adic representation of the number $x$. 

\begin{definition}
We define the \emph{frequency (asymptotic frequency) of the digit $i$ in the $s$--adic representation of a number $x \in [0,1]$} as the limit
$$
\nu_i(x)=\lim\limits_{k\to\infty}\displaystyle\frac{N_i(x,k)}{k},
$$
provided that this limit exists.
\end{definition}

The frequency function $\nu_i(x)$ of the digit $i$ in the $s$--adic representation of a number $x \in [0,1]$ is well defined for $s$--adic irrational numbers, and for $s$--adic rational numbers, it becomes well defined once we adopt the convention above to use only the representation with period $(0)$.

Since the beginning of the 20th century researchers have actively studied the properties of sets of real numbers by analyzing the frequencies of digits in their representations in various numeral systems. É. Borel and H. Lebesgue produced pioneering results on normal and weakly normal numbers and A. S. Besicovitch, G. Eggleston, P. Billingsley, O. Y. Khinchin, M. M. Postnikov, I. J. Pyatetskyi--Shapiro, M. V. Pratsiovytyi, L. Olsen, G. M. Torbin and others investigated properties of sets defined through the frequencies of symbols or combinations of symbols in specific representations.

A well-known fact (Borel’s theorem \cite{Bor}) states that \emph{for almost all numbers in $[0,1]$ (in the sense of Lebesgue measure), the following equality holds}
$$
\nu_0(x)=\nu_1(x)=\ldots=\nu_{s-1}(x)=\displaystyle\frac{1}{s}.
$$
We call such numbers \emph{weakly normal or normal on the base $s$}. The set of non-normal numbers (those that are not normal) has Lebesgue measure zero and, therefore, may exhibit a fractal structure. A continuous class of fractal subsets of this set arises from the Besicovitch–Eggleston sets, defined as 
$$
E\equiv E[\tau_0,\tau_1,\ldots,\tau_{s-1}]=\{x:\nu_i(x)=\tau_i,\,\,i=0,1,\ldots,s-1\}.
$$ 
We compute the Hausdorff–Besicovitch fractal dimension of the set $E$ using the formula
$$
\alpha_0(E)=-\displaystyle\frac{\ln\tau_0^{\tau_0}\tau_1^{\tau_1}\ldots\tau_{s-1}^{\tau_{s-1}}}{\ln s}.
$$

As established in \cite{AlPrTor} and \cite{PrTorb}, the set of non-normal and essentially non--normal numbers (numbers for which at least one or all digit frequencies do not exist) forms a superfractal set, meaning that its Hausdorff–Besicovitch dimension equals 1.

\begin{definition}
We define the \emph{asymptotic mean of digits} (or simply the \emph{mean of digits}) of a number $x \in [0,1]$ as the number
$$
r(x)\equiv\lim\limits_{n\to\infty}\frac{1}{n}\sum^{n}_{i=1}\alpha_i(x),
$$
where $A \ni \alpha_i$ denotes the $i$--th digit of the $s$--adic representation of the number $x$ (provided that the above limit exists).
\end{definition}

The asymptotic mean of the digits serves as a certain analogue of the frequency of an $s$--adic digit. Moreover, for $s=2$ the equality $r(x) = \nu_1(x)$ holds. This equality also holds for any number whose $s$--adic representation contains only the digits ``0'' and ``1'' or even a finite quantity of other digits. Therefore, for such numbers the absence of a frequency for the digit ``1'' (and one can easily construct an example of such a number) is equivalent to the absence of the asymptotic mean of the digits.

We focus on sets of numbers with a prescribed asymptotic mean of digits, that is, sets of the form
$$
S_a=\left\{x:r(x)=\lim_{n\to\infty}\frac{1}{n}\sum^{n}_{i=1}\alpha_i(x)=a\geqslant 0\right\}.
$$
Here, the constant $a$ represents a given parameter from $[0,2]$, and we are interested in the topological, metric, and fractal properties of these sets. In particular, in this paper we study the topological, metric, and fractal properties of sets of numbers represented in the ternary ($3$--adic)numeral  system that have a prescribed asymptotic mean of digits and for which the frequencies of all ternary digits exist.

\section{Relationship between digit frequencies and the asymptotic mean of digits}

\begin{lemma}
If the base $s$ of the numeral system satisfies $s>2$ and $\nu_i(x)=0$ for all $i>1$ then $r(x)$ and $\nu_1(x)$ either both do not exist or both exist and in the latter case $\nu_1(x) = r(x)$.
\end{lemma}
\begin{proof}
Since
$
  \dfrac{1}{n}\sum\limits^{n}_{i=1}\alpha_i(x)=\displaystyle\frac{0\cdot N_0(x,n)}{n}+\displaystyle\frac{1\cdot N_1(x,n)}{n}+\ldots+\displaystyle\frac{(s-1)\cdot N_{s-1}(x,n)}{n}
$
and $\nu_2=\nu_3=\ldots=\nu_{s-1}=0$, then
$$
  \displaystyle\frac{1}{n}\sum^{n}_{i=1}\alpha_i(x)=\displaystyle\frac{1\cdot N_1(x,n)}{n}.
$$
By taking the limit, we obtain
\begin{center}
$
\lim\limits_{n\to\infty}\displaystyle\frac{1}{n}\sum^{n}_{i=1}\alpha_i(x)=\lim\limits_{n\to\infty}\displaystyle\frac{N_1(x,n)}{n},\:\: \text{hence}\:\:\:r(x)=\nu_1(x),
$
\end{center} provided that the above limit exists.
\end{proof}

\begin{lemma}
If the ternary representation of a number $x$ has frequencies for all digits, then it has an asymptotic mean of digits $r(x)$, and in this case, $r(x) = \nu_1(x) + 2\nu_2(x)$.
\end{lemma}

\begin{proof}
Since
$$
  \displaystyle\frac{1}{n}\sum^{n}_{i=1}\alpha_i(x)=
    \displaystyle\frac{0\cdot N_0(x,n)}{n}+\displaystyle\frac{1\cdot N_1(x,n)}{n}+\displaystyle\frac{2\cdot N_2(x,n)}{n} =\displaystyle\frac{N_1(x,n)}{n}+2\displaystyle\frac{N_2(x,n)}{n},
$$
then
$$
\lim\limits_{n\to\infty}\frac{1}{n}\sum^{n}_{i=1}\alpha_i(x)=\lim\limits_{n\to\infty}\left(\frac{N_1(x,n)}{n}+2\frac{N_2(x,n)}{n}\right)= \lim\limits_{n\to\infty}\frac{N_1(x,n)}{n}+2\lim\limits_{n\to\infty}\frac{N_2(x,n)}{n},
$$
provided that these limits exist.
From the last equality, we have $$r(x)=\lim\limits_{n\to\infty}\displaystyle\frac{1}{n}\sum^{n}_{i=1}\alpha_i(x)= \nu_1(x)+2\nu_2(x).$$
\end{proof}

\begin{lemma}\label{lema5.1.2}
If the ternary representation of a number does not have the frequency of one digit, the frequency of at least one other digit does not exist either.
\end{lemma}

\begin{proof}
Suppose that the frequency $\nu_k(x_0)$ does not exist, that is the limit
$\lim\limits_{n\to\infty}\frac{N_k(x_0,n)}{n}$ does not exist. Since $\frac{N_k(x_0,n)}{n}=1-\frac{N_j(x_0,n)}{n}-\frac{N_m(x_0,n)}{n}$ then the limit $\lim\limits_{n\to\infty}\left(\frac{N_j(x_0,n)}{n}+\frac{N_m(x_0,n)}{n}\right) $ does not exist. This means that the limits $\lim\limits_{n\to\infty}\frac{N_{j}(x_0,n)}{n}$ and $\lim\limits_{n\to\infty}\frac{N_{m}(x_0,n)}{n}$ either do not exist simultaneously, or the limit $\lim\limits_{n\to\infty}\frac{N_j(x_0,n)}{n}$ exists while the limit $\lim\limits_{n\to\infty}\frac{N_m(x_0,n)}{n}$ does not exist and $j,k,m=0,1,2$, $j\neq k\neq m$.
\end{proof}

\begin{lemma}\label{lema5.1.2}
If the frequency $\nu_i(x)$ exists for the ternary representation of a number $x$, then $\nu_i(x)=\nu_{2-i}(1-x)$, where $i\in\{0,1,2\}$.
\end{lemma}

\begin{proof}
If $x=\Delta^3_{\alpha_1\alpha_2\ldots\alpha_k\ldots}$ then $1-x=\Delta^3_{(2-\alpha_1)(2-\alpha_2)\ldots(2-\alpha_k)\ldots}$, hence $\alpha_k(1-x)=2-\alpha_k(x)$. Thus $N_i(x,k)=N_{2-i}(1-x,k)$, and therefore we have $\nu_i(x)=\nu_{2-i}(1-x).$
\end{proof}

\begin{theorem}\label{teorema5.1.1}
If the ternary representation of a number $x$ is periodic with period $(s_1 \ldots s_m)$, then it has frequencies for all digits, and moreover
$$
\nu_i(x)=\displaystyle\frac{m_i}{m},
$$
where $m_i$ denotes the quantity of digits $i$ in the sequence $(s_1 \ldots s_m)$.
\end{theorem}

\begin{proof}
Since the frequency of a digit does not depend on any finite set of first digits, it suffices to carry out the proof for the number  $x_0=\Delta^3_{(s_1\ldots s_m)}$.

Let $b_n=n\cdot m$. Then
$$
N_0(x_0,b_n)=n\cdot m_0,\,\,N_1(x_0,b_n)=n\cdot m_1,\,\,N_2(x_0,b_n)=n\cdot m_2,
$$
and
$$
\displaystyle\frac{N_i(x_0,b_n)}{b_n}=\displaystyle\frac{n\cdot m_i}{n\cdot m}=\displaystyle\frac{m_i}{m},\,\,\,i=0,1,2.
$$

Let $k$ be an arbitrary natural number with $k > m$, then there exists an $n = n(k)$ such that
$$
b_n\leqslant k<b_{n+1}.
$$ Hence
$$
N_i(x_0,b_n)\leqslant N_i(x_0,k)\leqslant N_i(x_0,b_{n+1}),
$$ and
$$
\displaystyle\frac{n\cdot m_0}{(n+1)\cdot m}
=\displaystyle\frac{N_i(x_0,b_n)}{b_{n+1}}<\displaystyle\frac{N_i(x_0,k)}{k}\leqslant\displaystyle\frac{N_i(x_0,b_{n+1})}{b_n}
=\displaystyle\frac{(n+1)\cdot m_0}{n\cdot m}.
$$

Since $$\lim_{n\to\infty}\displaystyle\frac{n+1}{n}=\lim_{n\to\infty}\displaystyle\frac{n}{n+1}=1,$$ Then we have
$$\nu_i(x)=\lim\limits_{k\to\infty}\displaystyle\frac{N_i(x_0,k)}{k} =\displaystyle\frac{m_i}{m}.$$
\end{proof}

\begin{corollary}\label{naslidok5.1.1}
Each rational number in the interval $[0,1]$ has frequencies for all ternary digits.
\end{corollary}

\begin{corollary}\label{naslidok5.1.2}
A number $x$ whose ternary representation is periodic with period $(s_1 \ldots s_m)$ always has an asymptotic mean of digits.
\end{corollary}

\section{Frequency function of a ternary digit}

\begin{theorem}(Properties of the digit frequency function).\label{teorema5.1.2} 
\begin{enumerate}
  \item The digit frequency function $\nu_i(x)$, $i \in \mathcal{A}$, is bounded and attains all values in the interval $[0,1]$.
  \item The set of points where the digit frequency function is undefined is dense in $[0,1]$.
  \item The digit frequency function $\nu_i(x)$, $i \in \mathcal{A}$, is discontinuous at every point in $[0,1]$.
\end{enumerate}
\end{theorem}

\begin{proof}
1. Since $k^{-1} N_i(x,k) \leq 1$, then the function $\nu_i(x)$ cannot take values greater than 1, and therefore it is bounded.

Let us specify a number $x \in [0,1]$ for which $\nu_i(x) = a \in [0,1]$.
If $a = 1$, then such a number is $x = \Delta^3_{(i)}$, where $i \in {0,1,2}$.
If $a = 0$, then $x = \Delta^3_{(j)}$, where $j \in {0,1,2}$ and $j \neq i$.

Now let $0 < a < 1$. If $a$ is a rational number, that is $a=\displaystyle\frac{l}{m},$ then the frequency of the $i$--th ternary digit of the number
 $x=\Delta^3_{(\underbrace{i\ldots i}_l\underbrace{j\ldots j}_{m-l})},$
is equal to $a$.

If $a$ is irrational, we define the integer sequence $(c_n) = [n \cdot a]$.
Since $0<a<1$ then $c_1=0$. Let us consider the difference \\
$d_n=c_{n+1}-c_n=[(n+1)\cdot a]-[n\cdot a]=[n\cdot a+a]-[n\cdot a]=[[n\cdot a]+\{n\cdot a\}+a]-[n\cdot a].$\\ Since $[[A]+[B]]=[A]+[B]$, then $$d_n=[n\cdot a]+[\{n\cdot a\}+a]-[n\cdot a]=[\{n\cdot a\}+a].$$ Taking into account that $\{n\cdot a\}\in[0;1)$ і $a\in(0;1)$ we have $d_n\in\{0,1\}$.

Let us consider the number $x_*$ such that $N_i(x_*,n)=c_n$. Clearly such a number exists. Indeed, one can take  $x_*=\Delta^3_{\alpha_1\alpha_2\ldots\alpha_n\ldots}$, where
$$
 \alpha_{n+1}=
               \begin{cases}
                  i, & \text{if $d_n=1$,} \\
                  j, & \text{if $d_n=0$,}
               \end{cases}
$$
and $j\in \{0,1,2\}\setminus \{i\}$.
Since $a-\displaystyle\frac{1}{n}=\displaystyle\frac{n\cdot a-1}{n}<\displaystyle\frac{N_i(x_*,n)}{n}=\displaystyle\frac{[n\cdot a]}{n}\leqslant\displaystyle\frac{n\cdot a}{n}=a$, then $$\nu_i(x_*)=\lim_{n\to\infty}\displaystyle\frac{N_i(x_*,n)}{n}=a.$$

2. There exist numbers that do not have the frequency of at least two ternary digits. For example, such a number is $$x^*=\Delta^3_{ijiijjiiiijjjj...\underbrace{i...i}_{2^n}\underbrace{j...j}_{2^n}...} $$
Indeed, let us consider the sequence $(k_n)$, where $k_n$ represents the position of the last occurrence of the digit $i$ in the $(n+1)$-th block of consecutive $i$’s. \\
$N_i(x,k_n)=1+2+2^2+\ldots+2^n=\displaystyle\frac{2^{n+1}-1}{2-1}=2^{n+1}-1,$\\
$N_j(x,k_n)=1+2+2^2+\ldots+2^{n-1}=\displaystyle\frac{2^n-1}{2-1}=2^n-1,$\\
$\begin{array}{ll}
   k_n & =(1+2+\ldots+2^n)+(1+2+\ldots+2^{n-1})=2(1+2+\ldots+2^{n-1})+2^n=\\
       & =2\cdot\displaystyle\frac{2^n-1}{2-1}+2^n=2^{n+1}+2^n-2;
\end{array}$
$$\displaystyle\lim_{n\to\infty}\frac{N_i(x,k_n)}{k_n}=\displaystyle\lim_{n\to\infty}\frac{2^{n+1}-1}{2^{n+1}+2^n-2}=\displaystyle\frac{2}{3},$$
$$\displaystyle\lim_{n\to\infty}\frac{N_j(x,k_n)}{k_n}=\displaystyle\lim_{n\to\infty}\frac{2^n-1}{2^{n+1}+2^n-2}=\displaystyle\frac{1}{3}.$$
Now let $k'_n$ denote the position of the last digit $j$ in the $(n+1)$-th block of consecutive $j$’s. Clearly, \\
$N_i(x,k'_n)=1+2+2^2+\ldots+2^n=\displaystyle\frac{2^{n+1}-1}{2-1}=2^{n+1}-1=N_j(x,k'_n),\\
 k'_n=2\cdot(1+2+2^2+\ldots+2^n)=2(2^{n+1}-1)=2^{n+2}-2;$\\
$$\displaystyle\lim_{n\to\infty}\frac{N_i(x,k'_n)}{k'_n}=\displaystyle\lim_{n\to\infty}\frac{N_j(x,k'_n)}{k'_n}=\displaystyle\lim_{n\to\infty}\frac{2^{n+1}-1}{2^{n+2}-2}=\displaystyle\frac{1}{2},\,\,\,i,j\in\{0,1,2\}.$$\\
Hence $$\displaystyle\lim_{n\to\infty}\frac{N_i(x,k_n)}{k_n}=\displaystyle\frac{2}{3}\neq\displaystyle\frac{1}{2}=\displaystyle\lim_{n\to\infty}\frac{N_i(x,k'_n)}{k'_n},$$ $$\displaystyle\lim_{n\to\infty}\frac{N_j(x,k_n)}{k_n}=\displaystyle\frac{1}{3}\neq\displaystyle\frac{1}{2}=\displaystyle\lim_{n\to\infty}\frac{N_j(x,k'_n)}{k'_n}.$$
This demonstrates that the number $x$ has neither the frequency of the digit $i$ nor the frequency of the digit $j$.

Since the frequency does not depend on any finite number of initial digits in a number’s ternary representation, then every number in the tail set $x^*$ lacks the frequency of both digits $i$ and $j$. Since this tail set is dense in $[0,1]$, then the digit frequency function $\nu_i(x)$ exhibits a discontinuity within every arbitrarily small interval.

3. If the frequency $\nu_i(x_0)$ does not exist, then the function $\nu_i$ is discontinuous at the point $x_0$. Since the definition of continuity at a point requires the function to be defined in some neighborhood of that point, then it follows from Proposition 2 that the function $\nu_i$ is discontinuous at every point in $[0,1]$.
\end{proof}

The following property of the function $\nu_i$ provides a refinement of its discontinuity.

\begin{lemma}
The limit $\lim\limits_{x \to x_0} \nu_i(x)$ does not exist, even when the frequency $\nu_i(x_0)$ exists.
\end{lemma}
\begin{proof}
We provide a constructive proof of this fact.

If $\nu_i(x_0)\neq\frac{1}{3}$ then let us consider the sequence $(x_n)$: 
$$x_n=\Delta^3_{\alpha_1(x_0)...\alpha_n(x_0)(012)}.$$ 
According to Theorem 1 we have $\nu_i(x_0)=\frac{1}{3}$. It is clear that $x_n\to x_0$ as $n\to\infty$ but $$\lim_{x_n\to x_0}\nu_i(x_n)=\frac{1}{3}\neq\nu_i(x_0).$$

If $\nu_i(x_0)=\frac{1}{3}$ then let us consider the sequence
$$x'_n=\Delta^3_{\alpha_1(x_0)...\alpha_n(x_0)(i)},$$
hence we have
\begin{eqnarray}
& \lim\limits_{n\to\infty}\nu_i(x'_n)=1=\lim\limits_{x_n\to x_0}\nu_i(x'_n)\neq\nu_i(x_0). \label{af241}\nonumber&
\end{eqnarray}

Now let us consider the case when the frequency $\nu_i(x)$ in the representation of the number $x$ does not exist. Since $x_n\to x_0$ as $n\to\infty$ and $x'_n\to x_0$  as $n\to\infty$ but also
$$\frac{1}{3}=\lim\limits_{n\to\infty}\nu_i(x_n)\neq \lim\limits_{n\to\infty}\nu_i(x'_n)=1,$$
therefore, the limit of the frequency function $\lim\limits_{x \to x_0} \nu_i(x)$ does not exist.
\end{proof}

\section{The set of numbers with a prescribed asymptotic mean of digits}
We consider the set $S_a=\left\{x:r(x)=a\geqslant 0\right\}$ 
of numbers with a prescribed asymptotic mean of digits equal to $a$, where $0 \leq a \leq 2$.

\begin{lemma}
If the asymptotic mean of digits $r(x)$ and the frequency $\nu_0(x)$ exist, then the frequencies $\nu_1(x)$ and $\nu_2(x)$ also exist. Conversely, if the frequency $\nu_0(x)$ does not exist but $r(x)$ exists, then the frequencies $\nu_1(x)$ and $\nu_2(x)$ do not exist either.
\end{lemma}
\begin{proof}
Let $v^{(n)}_j=n^{-1}N_j(x,n)$ denote the relative frequency of the digit $j$ in the ternary representation of a number $x$, and $r_n(x)=\frac{1}{n}\sum\limits^{n}_{j=1}\alpha_j(x)$ denote the relative mean of the digits of $x$. Then the following system of equations holds:
$$
\begin{cases}
v^{(n)}_0+v^{(n)}_1+v^{(n)}_2=1,\\
v^{(n)}_1+2v^{(n)}_2=r_n.\\
\end{cases}
$$
We rewrite it in the form
$$
  \begin{cases}
    v^{(n)}_2=r_n-1+v^{(n)}_0,\\
    v^{(n)}_1=2-2v^{(n)}_0-r_n.\\
  \end{cases}
$$

If the limits $\lim\limits_{n\to\infty}r_n$ and $\lim\limits_{n\to\infty}v^{(n)}_0$ exist, then the second equation of the system implies the existence of $\lim\limits_{n\to\infty}v^{(n)}_1$ and the first equation implies the existence of $\lim\limits_{n\to\infty}v^{(n)}_2$.

If the limit $\lim\limits_{n\to\infty}r_n$ exist and the limit $\lim\limits_{n\to\infty}v^{(n)}_0$ does not exist then the limits $\lim\limits_{n\to\infty}v^{(n)}_1$ and $\lim\limits_{n\to\infty}v^{(n)}_2$ do not exist, therefore the frequencies $\nu_1(x)$ and $\nu_2(x)$ do not exist.
\end{proof}

The set $S_a$ is the union of two disjoint sets:
$K_a=\left\{x:\nu_1(x)\,\,\,\text{and}\,\,\,\nu_2(x)\,\,\,\text{exist}\right\}$, $M_a=\left\{x:\nu_0(x)\,\,\,\text{does not exist}\right\}$.

Let us analyze the properties of the subset $K_a$ of the set $S_a$.

\begin{theorem}
The set $K_a$ is continuum, everywhere dense set in the interval $[0;1]$. It has zero Lebesgue measure when $a \neq 1$ and full Lebesgue measure when $a = 1$. Its Hausdorff--Besicovitch dimension is at least equal to the number
$$
-\displaystyle\frac{1}{\ln 3}\ln\Biggl[ \left(\displaystyle\frac{t-1}{3}\right)^\frac{t-1}{3} \left(\displaystyle\frac{7-3a-t}{6}\right)^\frac{7-3a-t}{6} \left(\displaystyle\frac{1+3a-t}{6}\right)^\frac{1+3a-t}{6}\Biggr],
$$
where $t=\sqrt{1+6a-3a^2}.$
\end{theorem}

\begin{proof}
1. \emph{Continuum property.}
Since, for any point $x$ in the set $K_a$ the equality $r(x)=\nu_1(x)+2\nu_2(x)$ holds, the set $K_a$ contains all Besicovitch--Eggleston sets $E\equiv E[\tau_0,\tau_1,\tau_2]$ for which $$\tau_1+2\tau_2=a.$$ Therefore, the continuum property of $K_a$ follows from the continuum property of the sets $E$.

We provide a direct proof of this fact.
Let us define an injective mapping $f$ from the interval $[0;1]$ into the set $E$, that is, a mapping for which $x_1 \neq x_2$ implies $f(x_1) \neq f(x_2)$.

If $\tau_i = 1$, $i \in {0,1,2}$ then $x_0 = \Delta^3_{(i)} \in E$. In this case the mapping can be taken as
$$
f(x)=f(\Delta^3_{\alpha_1\alpha_2\ldots\alpha_n\ldots})=\Delta^3_{\beta_1\beta_2\ldots\beta_n\ldots},
$$
where $\beta_{3^n}=\alpha_n$, $n=1,2,\ldots$, and $\beta_j=i$ when $j\not\in \{3^n\}.$

If $\tau_i\in[0;1)$ then the number $x_0=\Delta^3_{c_1c_2\ldots c_n\ldots}$ is ternary--irrational, that is, it contains an infinite number of blocks of zeros, ones, and twos, and thus has the form
$$x_0=\Delta^3_{\underbrace{0\ldots0}_{a_1}\underbrace{1\ldots1}_{a_2}\underbrace{2\ldots2}_{a_3}\ldots\underbrace{0\ldots0}_{a_{3k-2}}\underbrace{1\ldots1}_{a_{3k-1}}\underbrace{2\ldots2}_{a_{3k}}\ldots},$$ where $a_i\in N_0.$
To an arbitrary $x=\Delta^3_{d_1d_2\ldots d_n\ldots}$ from the interval $[0;1]$ we assign $\widehat{x}=f(x)=x_0*x,$ where
$$\widehat{x}=\Delta^3_{\underbrace{0\ldots0}_{a_1}\underbrace{1\ldots1}_{a_2}\underbrace{2\ldots2}_{a_3} \alpha_1 \underbrace{0\ldots0}_{a_4}\underbrace{1\ldots1}_{a_5}\underbrace{2\ldots2}_{a_6}\underbrace{0\ldots0}_{a_7}\underbrace{1\ldots1}_{a_8}\underbrace{2\ldots2}_{a_9}\alpha_2\underbrace{0\ldots0}_{a_{10}}\ldots}$$\\
that is, $f(x)=\Delta^3_{\beta_1\beta_2\ldots\beta_n\ldots},$ where $$\beta_{s_{3^n}+n}(\widehat{x})=d_n,\,\,n=1,2,\ldots ,\,\,\beta_j(\widehat{x})=c_k(x_0)\,\,\text{when}\,\,j\not\in(s_{3^n}+n).$$

Let us show that $\widehat{x}\in E$, namely: $\nu_0(\widehat{x})=\nu_0(x_0),\,\,\nu_1(\widehat{x})=\nu_1(x_0)$. To do this, it suffices to show that this holds even in the case when $x=\Delta^3_{(02)}.$

We introduce the following notation: let $a_{3k-2}$ denote the length of the $k$-th block of zeros, $a_{3k-1}$ the length of the $k$-th block of ones, and $a_{3k}$ the length of the $k$-th block of twos in the number $x_0$.
\[
  \begin{array}{ll}
    u_k&=a_1+a_4+\ldots +a_{3k-2}, \\
    v_k&=a_2+a_5+\ldots +a_{3k-1}, \\
    w_k&=a_3+a_6+\ldots +a_{3k}.
  \end{array}
\]

Then the sum of the lengths of the first $k$ blocks of ones in the number $\widehat{x}$ equals $v_k$, while the sums of the lengths of the first $k$ blocks of zeros and twos are, respectively equals to
\begin{eqnarray}
& u'_k=a_1+(a_4+1)+a_7+a_{10}+\ldots +(a_{3^{2i-1}+1}+1)+\ldots +a_{3k-2}=u_k+N, \label{af241}\nonumber&\\
& w'_k=a_3+a_6+(a_9+1)+\ldots +(a_{3^{2i}+1}+1)+\ldots +a_{3k}=w_k+N.\nonumber&
\end{eqnarray}

Let us calculate the values of the frequencies of the digits zero and two for the numbers $x_0$ and $\widehat{x}$.\\
$\begin{array}{ll}
\nu_0(x_0)& =\lim\limits_{k\to\infty}\displaystyle\frac{u_k}{u_k+v_k+w_k}=\lim\limits_{k\to\infty}\displaystyle\frac{1}{1+\frac{v_k}{u_k}+\frac{w_k}{u_k}},
\end{array}$\\
$\begin{array}{ll}
\nu_2(x_0)& =\lim\limits_{k\to\infty}\displaystyle\frac{w_k}{u_k+v_k+w_k}=\lim\limits_{k\to\infty}\displaystyle\frac{1}{1+\frac{u_k}{w_k}+\frac{v_k}{w_k}},
\end{array}$\\
$\begin{array}{ll}
\nu_0(\widehat{x})& =\lim\limits_{k\to\infty}\displaystyle\frac{u'_k}{u'_k+v'_k+w'_k}= \lim\limits_{k\to\infty}\displaystyle\frac{u_k+N}{u_k+v_k+w_k+2N}=\\
                   &=\lim\limits_{k\to\infty}\displaystyle\frac{1}{1+\frac{v_k}{u_k}+\frac{w_k}{u_k}+\frac{2N(k)}{u_k}}+\lim\limits_{k\to\infty}\displaystyle\frac{1}{\frac{u_k}{N}+\frac{v_k}{N}+\frac{w_k}{N}+2},\\
\nu_2(\widehat{x})& =\lim\limits_{k\to\infty}\displaystyle\frac{w'_k}{u'_k+v'_k+w'_k}= \lim\limits_{k\to\infty}\displaystyle\frac{w_k+N}{u_k+v_k+w_k+2N}=\\
                   &=\lim\limits_{k\to\infty}\displaystyle\frac{1}{1+\frac{u_k}{w_k}+\frac{v_k}{w_k}+\frac{2N}{w_k}}+\lim\limits_{k\to\infty}\displaystyle\frac{1}{\frac{u_k}{N}+\frac{v_k}{N}+\frac{w_k}{N}+2}.\\
\end{array}$\\
Since\\
$\lim\limits_{k\to\infty}\displaystyle\frac{u_k}{N(k)}\geqslant\lim\limits_{k\to\infty}\displaystyle\frac{3^{2k-1}+1}{k}= \lim\limits_{x\to\infty}\displaystyle\frac{3^{2x-1}+1}{x}=\lim\limits_{x\to\infty}\displaystyle\frac{2\cdot 3^{2x-1}\ln3}{1}=\infty$, and\\
$\lim\limits_{k\to\infty}\displaystyle\frac{w_k}{N(k)}\geqslant\lim\limits_{k\to\infty}\displaystyle\frac{3^{2k}}{k}= \lim\limits_{x\to\infty}\displaystyle\frac{3^{2x}}{x}=\lim\limits_{x\to\infty}\displaystyle\frac{2\cdot 3^{2x}\ln3}{1}=\infty$, then\\
$\nu_0(\widehat{x})=\lim\limits_{k\to\infty}\displaystyle\frac{1}{1+\frac{v_k}{u_k}+\frac{w_k}{u_k}}$ and $\nu_2(\widehat{x})=\lim\limits_{k\to\infty}\displaystyle\frac{1}{1+\frac{u_k}{w_k}+\frac{v_k}{w_k}}$, that is,
$$
\nu_0(\widehat{x})=\nu_0(x_0),\,\,\nu_1(\widehat{x})=\nu_1(x_0).
$$
Therefore, the Besicovitch--Eggleston set $E[\tau_0, \tau_1, \tau_2]$ has the cardinality of the continuum.
Since $E[\tau_0, \tau_1, \tau_2] \subseteq K_a$ then the set $K_a$ is a continuum set.

2. \emph{Everywhere density.} The set $K_a$ is everywhere dense set since in any cylinder $\Delta^3_{c_1\ldots c_k}=\{x:\alpha_j(x)=c_j,\:j=\overline{1,m}\}$ there exist points from this set. Specifically, if $x_0= \Delta^3_{\alpha_1\alpha_2\ldots\alpha_n\ldots}\in K_a$, then the point $x=\Delta^3_{c_1\ldots c_m\alpha_1\ldots\alpha_n\ldots}\in K_a \cap \Delta^3_{c_1\ldots c_m}$.

3. \emph{Lebesgue measure.} If $a=1$ then the set $K_a$ contains the set $H$ of normal on the base 3 numbers, that is, the set of numbers whose digit frequencies are equal:
$$\nu_0(x)=\nu_1(x)=\nu_2(x)=\displaystyle\frac{1}{3}.$$ Then
$$\nu_1(x)+2\nu_2(x)=\displaystyle\frac{1}{3}+2\cdot\displaystyle\frac{1}{3}=1.$$
Since $\lambda(H)=1$ then $\lambda(K_a)=1$.

If $a\neq1$ then $K_a\cap H=\varnothing$. Hence, $K_a\subset[0;1]\setminus H$,
$\lambda(K_a)=\lambda([0;1]\setminus H)=0$.

4. \emph{The Hausdorff–Besicovitch dimension} possesses the property of monotonicity.
$$
  E_1\subset E_2 \Rightarrow \alpha_0(E_1)\leqslant\alpha_0(E_2)
$$
and countable stability.
$$
  \alpha_0(\bigcup_n E_n)=\sup\limits_n\alpha_0(E_n).
$$
Thus,  $$\alpha_0(K_a)=\sup\limits_{\substack{
\nu_1+2\nu_2=a, \\
\nu_1+2\nu_2+\nu_3=1
}}\alpha_0(E(\nu_0,\nu_1,\nu_2)).$$

From the conditions $\nu_0(x)+\nu_1(x)+\nu_2(x)=1$ and $\nu_1(x)+2\nu_2(x)=a$ we determine
$\nu_1=a-2\nu_2$ and $\nu_0=1-\nu_1-\nu_2=1-a+2\nu_2-\nu_2=1-a+\nu_2.$
To compute the Hausdorff--Besicovitch dimension, we analyze the function to find its maximum.
$$
\begin{array}{ll}
y&=-\displaystyle\frac{\ln  x^{x}\cdot(a-2x)^{a-2x}\cdot(1-a+x)^{1-a+x}}{\ln 3},
\end{array}
$$ where $x\in[0;\frac{a}{2}]$, that is,\\
$
\begin{array}{ll}
y&=\displaystyle\frac{\ln x^{x}+\ln (a-2x)^{a-2x}+\ln (1-a+x)^{1-a+x}}{-\ln 3},
\end{array}
$\\
$
\begin{array}{ll}
y'&=-\displaystyle\frac{1}{\ln3}\left(\ln x+1-2\ln(a-2x)-2\displaystyle\frac{a-2x}{a-2x}+\ln(1-a+x)+\displaystyle\frac{1-a+x}{1-a+x}\right)=\\
&=-\displaystyle\frac{1}{\ln 3}\left(\ln x+1-2\ln(a-2x)-2+\ln(1-a+x)+1\right)=\\ &=-\displaystyle\frac{1}{\ln 3}\left(\ln\displaystyle\frac{(1-a+x)x}{(a-2x)^2}\right),
\end{array}
$\\
$
\begin{array}{ll}
y'& =0,\:\:\:\displaystyle\frac{(1-a+x)x}{(a-2x)^2}=1,
\end{array}
$\\
\[
\begin{cases}
(1-a+x)x-(a-2x)^2=0,\\
x\neq\frac{a}{2};\\
\end{cases}
\]
$3x^2-(1+3a)x+a^2=0,\\
x_{1}=\displaystyle\frac{1+3a-\sqrt{1+6a-3a^2}}{6}\,\,\text{ is point of maximum},\\
x_{2}=\displaystyle\frac{1+3a+\sqrt{1+6a-3a^2}}{6}\,\,\text{ is point of minimum},\\$
$$y_{max}=-\displaystyle\frac{1}{\ln 3}\cdot\ln\Biggl[ \left(\displaystyle\frac{t-1}{3}\right)^\frac{t-1}{3}\cdot \left(\displaystyle\frac{7-3a-t}{6}\right)^\frac{7-3a-t}{6}\cdot \left(\displaystyle\frac{1+3a-t}{6}\right)^\frac{1+3a-t}{6}\Biggr],$$ where $t=\sqrt{1+6a-3a^2}.$
\end{proof}

\begin{corollary}
The Hausdorff–Besicovitch dimension of the set $K_1=\left\{x:r(x)=1\right\}$
is equal to 1, that is, $\alpha_0(K_1)=1.$
\end{corollary}

\end{document}